\documentclass[12pt]{article}
\usepackage{amsxtra,amssymb,amsmath,enumerate,bbm,hyperref,amsthm,color}

\newtheorem{lemma}{Lemma}
\newtheorem{proposition}{Proposition}
\newtheorem{theorem}{Theorem}

\newtheorem{remark}{Remark}
\newtheorem{definition}{Definition}

\newcommand{\kla}{\left ( }
\newcommand{\mer}{\right ) }
\newcommand{\noo}{\left \|}
\newcommand{\rrm}{\right \|}
\newcommand{\bet}{\left |}
\newcommand{\rag}{\right |}
\newcommand{\equa}{\begin{eqnarray*}}
\newcommand{\tion}{\end{eqnarray*}}
\newcommand{\pl}{\hspace{.1cm}}
\newcommand{\la}{\lambda}
\newcommand{\vare}{\varepsilon}
\newcommand{\sptext}[3]{\hspace{#1 em}\mbox{#2}\hspace{#3 em}}
\renewcommand{\P}{\mathbbm{P}}
\newcommand{\Q}{\mathbbm{Q}}
\newcommand{\R}{\mathbbm{R}}
\newcommand{\E}{\mathbbm{E}}

\newcommand{\rh}{\mathcal{RH}}
\newcommand{\bmo}{{\rm BMO}}
\def\cF{{\mathcal F}}
\def\cG{{\mathcal G}}

\def \ept{\E^{\cF_t}_\P}
\def \eqt{\E^{\cF_t}_\Q}

\def \equ{\E^{\cF_u}_\Q}
\allowdisplaybreaks

\begin{document}

\title{Fractional smoothness of functionals of diffusion processes under a change of measure}

\author{Stefan Geiss\footnote{This research is supported by the Project 133914 of the Academy of Finland.}
        \pl\pl
        and
        Emmanuel Gobet \footnote{This research is part of the Chair {\it Financial Risks} of the {\it Risk
Foundation}, the Chair {\it Derivatives of the Future} and the Chair {\it Finance and Sustainable Development}.} \\ \\
        Department of Mathematics  \\
        University of Innsbruck \\
        Technikerstra\ss e 13/7 \\
        A-6020 Innsbruck \\
        Austria \\
        stefan.geiss@uibk.ac.at \\ \\
        Centre de Math\'ematiques Appliqu\'ees and CNRS\\
        Ecole Polytechnique\\
        F-91128 Palaiseau Cedex\\
        France \\
        emmanuel.gobet@polytechnique.edu}
        
\maketitle

\begin{abstract}
Let $v:[0,T]\times \R^d \to \R$ be the solution of the parabolic backward equation
$\partial_t v + (1/2) \sum_{i,l} [\sigma \sigma^\perp]_{il} \partial_{x_i}\partial_{x_l} v
+ \sum_{i} b_i \partial_{x_i}v + kv =0$ with terminal condition
$g$, where the coefficients are time- and state-dependent, and satisfy certain regularity assumptions.
Let $X=(X_t)_{t\in [0,T]}$ be the associated $\R^d$-valued 
diffusion process on some appropriate $(\Omega,\cF,\Q)$. For $p\in [2,\infty)$ and
a measure $d\P=\lambda_T d\Q$, where $\lambda_T$ satisfies the Muckenhoupt condition $A_\alpha$ for $\alpha \in (1,p)$, 
we relate the behavior of 
$\|g(X_T)-\ept g(X_T) \|_{L_p(\P)}$,
$\|\nabla v(t,X_t)  \|_{L_p(\P)}$ and
$\|D^2 v(t,X_t)  \|_{L_p(\P)}$ to each other, where
$D^2v:=(\partial_{x_i}\partial_{x_l}v)_{i,l}$ is the Hessian matrix.

\end{abstract}
        

\section{Introduction}

For a fixed time-horizon $T>0$ let $(\Omega,\cF,(\cF_t)_{t\in [0,T]},\Q)$ be a filtered probability space where $(\Omega,\cF,\Q)$ is complete, $\cF=\cF_T$, the filtration $(\cF_t)_{t\in [0,T]}$ is right-continuous,
$\cF_0$ is generated by the null sets of $\cF$ and where all local martingales are continuous
(see Section \ref{sec:setting}).
Assume for some $d\ge 1$ that the process $B=(B_t)_{t\in [0,T]}$ is a $d$-dimensional 
$(\cF_t)_{t\in [0,T]}$-standard 
Brownian motion starting in zero.
We consider an $\R^d$-valued diffusion process $X=(X_t)_{t\in [0,T]}$, solution to the stochastic differential equation 
\equa
X_t  = x_0 + \int_0^t \sigma(s,X_s) dB_s + \int_0^t b(s,X_s) ds
\tion 
for some smooth bounded coefficients $b$ and $\sigma$, and we focus on the rate of convergence of 
\[ R^X_p(t):=\|g(X_T)-\E(g(X_T)|\cF_t)\|_p \]
for $p\in [2,\infty)$
as $t\rightarrow T$, where $g$ satisfies a suitable growth condition ensuring $g(X_T)\in L_p$. The behavior of $R^X_p(t)$ as $t\rightarrow T$ 
is a measure of the {fractional smoothness}  of $g$, see \cite{geis:gobe:11} for an overview. Actually it is now well-known 
\cite{geis:geis:04,geis:hujo:07,gobe:makh:10,geis:geis:gobe:12} that there is a precise correspondence between the irregularity of the 
terminal function $g$ and the time-singularity of 
the $L_p$-norms of $\nabla v(t,X_t)$ as $t\uparrow T$ where
\[ v(t,x)=\E (g(X_T) | X_t = x ). \]
The aim of this paper is to extend these quantitative equivalence results to situations where the $L_p$-norms are computed under different measures. 
The theory of probabilistic Muckenhoupt weights, developed as a counterpart to the 
deterministic ones from \cite{muck:72} and other papers, gives a natural way
to extend various martingale inequalities to equivalent measures, see  exemplary
\cite{izum:kaza:77,bona:lepi:79,kaza:94} and the references therein. 
A typical situation is a change of measure initiated by a 
Girsanov transformation, i.e. a change of the drift of $X$. Applying the results of this paper 
in this particular case, gives -without going into full details- the following:
if the process $Y$ differs from $X$ by another bounded drift and 
if $\theta\in(0,1)$, then we have
\begin{equation}
\label{eq:intro}
\sup_{t\in [0,T)}(T-t)^{-\theta/2} R^Y_p(t)<\infty
 \Longleftrightarrow  
\sup_{t\in [0,T)}(T-t)^{(1-\theta)/2}\|\nabla v(t,Y_t)\|_p<+\infty
\end{equation}
which follows from Theorem \ref{thm:equivalence} below for $q=\infty$ as explained in Remark \ref{remark:thm_equivalence}(7). The parameter 
$\theta$ is the degree of \emph{fractional smoothness}. 
\smallskip

Regarding the references in the literature related to \eqref{eq:intro}, a 1-dimensional diffusion case with $X=Y$ is considered 
in \cite{geis:geis:04}, the extension to multidimensional processes is performed in \cite{geis:hujo:07} 
in the case $X=Y$ being a Brownian motion and in \cite{gobe:makh:10} for diffusion processes. In \cite{geis:geis:gobe:12} path-dependent functionals are considered. For an overview the reader is referred
to \cite{geis:gobe:11}.
Actually our main result (Theorem \ref{thm:equivalence}) takes a more general form than \eqref{eq:intro}:
\begin{itemize}
\item we consider an $L_q([0,T),\frac{dt  }{T-t })$-norm with $q\in [2,\infty]$ instead of  the above 
      $L_\infty$-norm with respect to $t\in [0,T)$;
\item we consider an additional potential factor $k$ in our parabolic problem to define $v$;
\item the change of measure, described  in \eqref{eq:intro} by the change from $X$ to $Y$, 
      is described by Muckenhoupt weights;
\item we also state results regarding the second derivatives.
\end{itemize}

\paragraph{Applications.}\hspace*{-1em} The tight control of the behavior of the norms $\|\nabla v(t,X_t)\|_{L_2}$ as $t\rightarrow T$ is an issue that has been raised in \cite{geis:geis:04}, where the purpose was to analyze discrete approximations of stochastic integrals coming from the representation
\begin{equation}\label{eq:pred:rep}
g(X_T)=v(0,x_0)+\int_0^T \nabla v(t,X_t) \sigma(t,X_t) dB_t.
\end{equation}
Discretizing the above stochastic integral and analyzing the resulting approximation error in $L_2$,
requires a better understanding how strongly the irregularity of the terminal function $g$ transfers 
to the blow-up of the function $t \mapsto \|\nabla v(t,X_t)\|_{L_2}$ and higher derivatives of $v$ as well. 
Major consequences of this analysis are the derivation of tight convergence rates for uniform 
time grids and  the design of non-equidistant time grids to obtain optimal convergence rates.
\smallskip

Recently, similar results have been established in the context of Backward Stochastic Differential Equations 
\cite{gobe:makh:10,geis:geis:gobe:12} to pave the way for the development of more efficient numerical schemes.
\smallskip

Finally, similar issues arise in the analysis of the Delta-Gamma hedging strategies in Finance, 
which typically result in a higher order approximation of the stochastic integral \eqref{eq:pred:rep}, see \cite{gobe:makh:11}. 
\smallskip

Within the applications in Stochastic Finance
intrinsically two measures are involved: the historical measure for evaluating the risk, for example as
$L_p$-mean, and  
the risk-neutral measure, under which the price and the hedging strategy are computed and which is related 
to the above function $v$. 
For this setting, the current results are particularly of interest. Moreover, the potential $k$ may be interpreted as an interest rate.


\section{Setting}
\label{sec:setting}
\paragraph{Notation.} 
We denote by $|\cdot|$ the Euclidean norm of a vector.
Given a matrix $C$ considered as operator $C:\ell_2^n \to \ell_2^N$,
the expression $|C|$ stands for the Hilbert-Schmidt norm
and $C^\top$ for the transposed of $C$.
The $L_p$-norm ($p\in [1,\infty]$) of a random vector $Z:\Omega\to \R^n$ or
a random matrix $Z:\Omega\to\R^{n\times m}$ is denoted by $\|Z\|_p = \||Z|\|_{L_p}$.
As usual, $\partial^\alpha_x \varphi$ is the partial derivative of the 
order of an multi-index $\alpha$ (with length $|\alpha|$) with respect to $x\in \R^d$.
The Hessian matrix of a function $\varphi:\R^d\mapsto \R$ is abbreviated by $D^2 \varphi$
and the gradient (as row vector) by  $\nabla \varphi$. In particular, this means that 
$D^2$ and $\nabla$ always refer to the state variable $x\in\R^d$.
If we mention that a constant depends on $b$, $\sigma$ or $k$, then we implicitly indicate
a possible dependence on $T$ and $d$ as well. Finally, letting $h:[0,T]\times\R^d\to \R^{n\times m}$ we use
the notation $\| h\|_\infty := \sup_{t,x} |h(t,x)|$.

\paragraph{The parabolic PDE.} We fix $T>0$ and consider the Cauchy problem 
\equa 
\mathcal{L} v & = & 0 \sptext{3}{on}{.7} [0,T)\times\R^d, \\
     v(T,x)   & = & g(x)
\tion
with

\[      \mathcal{L}  
   :=  \partial_t 
      + \frac 12\sum_{i,j=1}^d a_{i,j}(t,x) \partial^{2}_{x_i,x_j}
      + \sum_{i=1}^db_i(t,x) \partial_{x_i} 
      + k(t,x), \]
where $A:=(a_{ij})_{ij}=\sigma\sigma^\top$. The assumptions on the coefficients and $g$ are as follows:

\begin{enumerate}[(C1)]
\item The functions $\sigma_{i,j},b_i,k$ are bounded and belong to  $C^{0,2}_b([0,T]\times \R^d)$ 
      and there is some $\gamma\in (0,1]$ such that the functions and their state-derivatives are
      $\gamma$-H\"older continuous with respect to the parabolic 
      metric on each compactum of $[0,T]\times\R^d$. Moreover, $\sigma$ is $1/2$-H\"older
      continuous in $t$ uniformly in $x$. \item $\sigma(t,x)$ is an invertible $d\times d$-matrix with $\sup_{t,x} |\sigma^{-1}(t,x)| < +\infty$;
\item the terminal function $g:\R^d\to \R$ is measurable and exponentially bounded: 
      for some $K_g\geq0$ and $\kappa_g\in[0,2)$ we have
      \[ |g(x)|\leq K_g\exp(K_g|x|^{\kappa_g}) \sptext{1}{for all}{1} x\in \R^d.\]
\end{enumerate}
The condition (C2) implies that there exists a $\delta>0$ with $\langle Ax,x \rangle \ge \delta |x|^2$
for all $x\in \R^d$, i.e. the operator $\mathcal{L}$ is uniformly parabolic.
Under the above assumptions there exists a fundamental solution:

\begin{proposition}[{\cite[Theorem 7, p. 260; Theorem 10, pp. 72-74]{frie:64}}]
\label{proposition:properties_Gamma}
Under the assumptions {\rm (C1)} and {\rm (C2)} there exists a fundamental solution 
$\Gamma(t,x;\tau,\xi): \{ 0\le t < \tau \le T\}\times \R^d\times \R^d\to [0,\infty)$
for $\mathcal{L}$ and a constant $c_{\eqref{eqn:theorem:properties_Gamma}}>0$ such that for $0\le |a|+2b\le 3$ the derivatives
$D^a_x D^b_t \Gamma$ exist in any order, are continuous, and satisfy
\begin{equation}\label{eqn:theorem:properties_Gamma}
             |D_x^{a} D_t^b \Gamma(t,x;\tau,\xi)|
         \le c_{\eqref{eqn:theorem:properties_Gamma}}
             (\tau-t)^{-\frac{|a|+2b}{2}} \gamma^d_{\tau-t} \kla 
             \frac{x-\xi}{ c_{\eqref{eqn:theorem:properties_Gamma}}} \mer
\end{equation}
where $\gamma_s^d(x) := e^{-\frac{|x|^2}{2s}}/(\sqrt{2\pi s})^d$.
\end{proposition}
\bigskip
For 
\equa
v(t,x) &:= & \int_{\R^d}\Gamma(t,x;T,\xi)g(\xi)d\xi,\\
v(T,x) &:= & g(x),
\tion
and $0\le |a|+2b\le 3$ Proposition \ref{proposition:properties_Gamma} implies that the derivatives 
$D_x^{a} D_t^b v$ exist in any order, are continuous on $[0,T)\times \R^d$  and satisfy 
\equa
\mathcal{L}v & = & 0 \sptext{1}{on}{1} [0,T)\times\R^d, \\
|D_x^{a} D_t^b v(t,x)|&\le& c (T-t)^{-\frac{|a|+2b}{2}} \exp(c |x|^{ \kappa_g})
\tion
for $x\in \R^d$ and $t\in[0,T)$, where $c>0$ depends at most on 
$(\kappa_g,K_g,c_\eqref{eqn:theorem:properties_Gamma},T)$.

\paragraph{The stochastic differential equation.} 
Let $(B_t)_{t\in [0,T]}$ be a $d$-dimensional standard Brownian motion
defined on $(\Omega,\cF,(\cF_t)_{t\in [0,T]},\Q)$, where $(\Omega,\cF,\Q)$ is complete,
$(\cF_t)_{t\in [0,T]}$ is right-continuous, $\cF=\cF_T$, 
$\cF_0$ is generated by the null sets of $\cF$ and where all local martingales are continuous.
\medskip

As we work on a closed time-interval we have to explain our understanding of a local martingale:
we require that the localizing sequence of stopping times $0\le \tau_1 \le \tau_2 \le \cdots \le T$ 
satisfies $\lim_n \Q(\tau_n=T)=1$. The reason for this is that we think about the extension of
the filtration constantly by $\cF_T$ to $(T,\infty)$ and that all local martingales
$(N_t)_{t\in [0,T]}$  (in our setting) are extended by $N_T$ to $(T,\infty)$. This yields the
standard notion of a local martingale.
However this is not needed explicitly in our paper, we only need this implicitly whenever
we refer to results about the Muckenhoupt weights $A_\alpha(\Q)$ from \cite{kaza:94}.
\medskip

To shorten the notation, we denote sometimes the conditional expectation $\E(.|\cF_t)$ by $\E^{\cF_t}(.)$.
The process $X=(X_t)_{t\in [0,T]}$ is given as strong unique solution of
\[ X_t  =  x_0 + \int_0^t \sigma(s,X_s) dB_s + \int_0^t b(s,X_s) ds. \]
Introducing the standing notation
\[ K_t^X := e^{\int_0^t k(r,X_r) dr}
   \sptext{1}{and}{1}
   M_t := K_t^X v(t,X_t), \]
It\^o's formula implies, for $t\in [0,T)$, that
\begin{eqnarray}
   M_t & = & v(0,x_0) + \int_0^t K_s^X \nabla v(s,X_s) \sigma(s,X_s) dB_s. 
\end{eqnarray}
Moreover, 
\begin{equation}\label{eqn:L_p-convergence_to_maturity}
   \lim_{t\to T} M_t = M_T 
   \sptext{1}{and}{1}
   \lim_{t\to T} v(t,X_t) = g(X_T)
\end{equation} 
almost surely and in any $L_r(\Q)$ with $r\in [1,\infty)$.
Using Proposition \ref{proposition:properties_Gamma} for $k=0$ we also have 
\[ \Q(|X_t-x_0|>\lambda) \leq c\exp\left (-\frac{\lambda^2}{c}\right ) \]
for all $\lambda\geq 0$ and  $t\in [0,T]$, where  $c>0$ depends at most on $(\sigma,b)$ and is,
in particular, independent from the starting value $x_0\in \R^d$. It directly implies 
that
\[ g(X_T) \in \bigcap_{r\in [1,\infty)} L_r(\Q) \]
so that Remark \ref{remark:L_p-inclusions} applies as well.
We will also use the following 

\begin{lemma}
[{\cite{gobe:muno:05}, 
  \cite[Proof of Lemma 1.1]{gobe:makh:10},
  \cite[Remark 3 in Appendix B]{geis:geis:gobe:12}}]
\label{lemma:Malliavin_weights}
Let $t\in (0,T]$, $h:\R^d\to \R$ be a Borel function satisfying {\rm (C3)} and
$\Gamma_X$ be the transition density of $X$, i.e. the function $\Gamma$ from
Proposition \ref{proposition:properties_Gamma} in the case $k=0$. Define
\[ H(s,x) := \int_{\R^d} \Gamma_X(s,x;t,\xi) h(\xi) d\xi
   \sptext{1}{for}{1}
   (s,x)\in [0,t)\times\R^d. \]
For $r\in [0,t)$ and $x\in \R^d$ let $(Z_u)_{u\in [r,t]}$ be the diffusion based on 
$(\sigma,b)$ starting in $x$ defined on some $(M,\cG,(\cG_u)_{u\in [r,t]},\mu)$ equipped with a 
standard $(\cG_u)_{u\in [r,t]}$-Brownian motion, where 
$(M,\cG,\mu)$ is complete, 
$(\cG_u)_{u\in [r,t]}$ is right-continuous and $\cG_r$ is generated by the null sets of $\cG$. 
Then, for $q\in (1,\infty)$ and $s\in [r,t)$, one has a.s. that
\equa
      |\nabla H(s,Z_s)| 
&\le& \kappa_q \frac{ [ \E( |h(Z_t)  - \E (h(Z_t)|\cG_s)|^q|\cG_s) 
      ]^\frac{1}{q}}{(t-s)^\frac{1}{2}}, \\
      |D^2 H(s,Z_s)| 
&\le& \kappa_q 
      \frac{[ \E( |h(Z_t)  - \E (h(Z_t)|\cG_s)|^q|\cG_s) ]^\frac{1}{q}}
           {t-s},
\tion
where $\kappa_q>0$ depends at most on $(\sigma,b,q)$.
\end{lemma}

\paragraph{Conditions on the equivalent measure.} 
In addition to the given measure $\Q$ we will use an equivalent measure $\P\sim\Q$ and agree about
the following standing assumption: 
\begin{enumerate}
\item [(P)]
      There exists a martingale $Y=(Y_t)_{t\in [0,T]}$ with $Y_0\equiv 0$ such that
      \[ \lambda_t := \mathcal{E}(Y)_t = e^{Y_t-\frac{1}{2}\langle Y \rangle_t} 
          \sptext{1}{for}{1}
          t\in [0,T] \]
      is a martingale and
      \[ d\P=\lambda_T d\Q.\]
\end{enumerate}

\begin{definition}
Assume that condition {\rm (P)} is satisfied.
\begin{enumerate}[{\rm (i)}]
\item For $\alpha\in (1,\infty)$ we say that $\lambda_T\in A_\alpha(\Q)$ provided that there is a 
      constant $c>0$ such that for all stopping times $\tau:\Omega\to [0,T]$ one has that
      \[ \E_\Q \kla \bet \frac{\lambda_\tau}{\lambda_T}\rag^\frac{1}{\alpha-1} | \cF_\tau \mer 
         \le c \quad\mbox{ a.s.} \]
\item For $\beta\in (1,\infty)$ we let  $\lambda_T\in \rh_\beta(\Q)$ provided that there is a 
      constant $c>0$ such that for all stopping times $\tau:\Omega\to [0,T]$ one has that
      \[     \E_\Q (|\la_T|^\beta |\cF_\tau)^\frac{1}{\beta}
         \le c \la_\tau \quad\mbox{ a.s.} \]
\end{enumerate}
\end{definition}
\smallskip
The class $A_\alpha(\Q)$ is the probabilistic variant of the Muckenhoupt condition and 
$\rh$ stands for {\em reverse H\"older inequality}. Next we need
\smallskip

\begin{definition}
A martingale $Z=(Z_t)_{t\in [0,T]}$ is called \bmo-martingale provided that
$Z_0\equiv 0$ and there is a $c>0$ such that for all stopping 
times $\tau:\Omega\to [0,T]$ one has that
\[ \E_\Q \kla |Z_T - Z_\tau|^2 |\cF_\tau \mer \le c^2\quad\mbox{ a.s.} \]
\end{definition}
\smallskip

It is known \cite[Theorems 2.3]{kaza:94} that 
$(e^{Z_t - \frac{1}{2} \langle Z \rangle_t})_{t\in [0,T]}$ is a martingale
provided that $Z$ is a \bmo-martingale.
\smallskip

\begin{proposition}[{\cite[Theorems 2.4 and 3.4]{kaza:94}}]
\label{proposition:RH-A-BMO}
Under condition {\rm (P)} the following assertions are equivalent:
\begin{enumerate}[{\rm (i)}]
\item $Y$ is a {\rm BMO}-martingale.
\item $\mathcal{E}(Y)\in A_\alpha(\Q)$ for some $\alpha\in (1,\infty)$.
\item $\mathcal{E}(Y)\in \rh_\beta(\Q)$ for some $\beta\in (1,\infty)$.
\end{enumerate}
\end{proposition}
\smallskip

\begin{remark}\label{remark:L_p-inclusions}
Under the assertions of Proposition \ref{proposition:RH-A-BMO}
we have that $\lambda_T\in L_\beta(\Q)$ and $1/\lambda_T\in L_{\alpha'}(\P)$ with
$1=(1/\alpha) + (1/\alpha')$ so that
\[ \bigcap_{r\in [1,\infty)} L_r(\Q) = \bigcap_{r\in [1,\infty)} L_r(\P). \]
\end{remark}
\smallskip

\begin{proposition}[{\cite[Theorems 2.3 and 3.19]{kaza:94}}]
\label{proposition:BMO-exponential}
Let $Y$ be a \bmo-martin\-gale  so that {\rm (P)} is satisfied.
For all $p\in (0,\infty)$ there is a $b_p(\P)>0$ such that for all $\Q$-martingales $N$ with $N_0\equiv 0$ one has that
\[     \frac{1}{b_p(\P)} \| N_T^* \|_{L_p(\P)}
   \le \| \sqrt{\langle N \rangle_T} \|_{L_p(\P)}
   \le b_p(\P) \| N_T^* \|_{L_p(\P)} \]
where $N^*_t:=\sup_{s \in [0,t]}|N_s|$.
\end{proposition}
\smallskip

\paragraph{An inequality.} Given a probability space $(M,\Sigma,\mu)$ with a sub-$\sigma$ algebra
$\cG\subseteq\Sigma$ and $Z\in L_p(M,\Sigma,\mu)$ with
$p\in [1,\infty]$ we have that
\begin{equation}
\label{eq:lp}
     \frac{1}{2}\,\|Z-\E(Z|\cG)\|_{p} 
\leq \inf_{Z'\in L_p(M,\cG,\mu)} \|Z-Z'\|_{p}
\leq \|Z-\E(Z|\cG)\|_{p}.
\end{equation}


\section{The result}

In the following $\theta\in (0,1]$ will be the main parameter of the fractional 
smoothness. Additionally, we introduce a fine-tuning parameter $q\in [2,\infty]$ 
and 
\[ \Phi_q (h) := \| h \|_{L_q\kla [0,T), \frac{dt}{T-t} \mer} \]
for a measurable function $h:[0,T)\to \R$.
The aim of this paper is to prove the following result:
\medskip

\begin{theorem}\label{thm:equivalence}
Let $p\in [2,\infty)$, $\alpha\in (1,p)$ and $\lambda_T\in A_\alpha(\Q)$, and 
assume that {\rm (C1)}, {\rm (C2)} and {\rm (P)} are satisfied. Then, for 
$\theta\in(0,1)$, $q\in [2,\infty]$, a measurable function $g:\R^d\to\R$ 
satisfying {\rm (C3)} and $d\P=\lambda_T d\Q$ the following assertions are equivalent:
\begin{enumerate}[{\rm (i$_{\theta}$)}]
\item $\Phi_q\kla (T-t)^{-\frac{\theta}{2}} \| g(X_T) - \E^{\cF_t}_\P g(X_T)\|_{L_p(\P)} \mer < +\infty$.
\item $\Phi_q\kla (T-t)^\frac{1-\theta}{2}  \| \nabla v(t,X_t) \|_{L_p(\P)}         \mer < +\infty$.
\item $\Phi_q\kla (T-t)^\frac{2-\theta}{2}  \| D^2 v(t,X_t) \|_{L_p(\P)}              \mer < +\infty$.
\end{enumerate}
\end{theorem}
\medskip 

\begin{remark}
\label{remark:thm_equivalence}
\rm
\begin{enumerate}[(1)]
\item Using \cite[Corollary 3.3]{kaza:94} it is sufficient to require that $\lambda_T\in A_p(\Q)$ 
      as in this case there is an $\vare\in (0,p-1)$ such that $\lambda_T\in A_{p-\vare}(\Q)$.
      One the other hand, it would be of interest to investigate the case when 
       $\lambda_T\in A_\alpha(\Q)$ with $\alpha > p$. This is not done here.

\item Examples of functions $g$ such that $(i_\theta)$ is satisfied are 
      given for example in \cite{geis:geis:04,geis:hujo:07,geis:toiv:09,geis:geis:gobe:12}.

\item In the case $X=B$, $\P=\Q$, $T=1$ and $k=0$ the conditions of 
      Theorem \ref{thm:equivalence} (neglecting the boundedness condition (C3))
      are equivalent   to that $g$ belongs to the  Malliavin Besov space $B_{p,q}^\theta$ on 
      $\R^d$ weighted by the standard Gaussian measure (see \cite{geis:toiv:12}).

\item \underline{The case $\theta=1$ and $q\in [2,\infty)$} is not considered 
      in Theorem \ref{thm:equivalence} because it yields to pathologies:
      Let $X=B$, $\P=\Q$, $T=1$ and $k=0$. Condition 
      ${\rm (i_1)}$ implies ${\rm (ii_1)}$ by Lemma \ref{lemma:H1} below.
      Moreover, condition ${\rm (ii_1)}$ and the monotonicity of 
      $\| \nabla v(t,B_t) \|_{L_p(\P)}$ 
      ($(\nabla v(t,B_t))_{t\in [0,1)}$ is a martingale in this case)  imply that  
      that $\nabla v(t,B_t)=0$ a.s. so that $g(B_1)$ is almost surely constant
      (for example, one can use 
      $g(B_1) = \E (g(B_1)) + \int_{(0,1]} \nabla v(t,B_t) dB_t)$. 

\item As the process $M=(M_t)_{t\in [0,T]}$ with $M_t=K_t^Xv(t,X_t)$ is a martingale under $\Q$ 
      it is natural to consider condition {\rm (i$_{\theta}$)} for the corresponding martingale under
      $\P$ as well:
      \begin{enumerate}[{\rm (i$'_{\theta}$)}]
      \item $\Phi_q\kla (T-t)^{-\frac{\theta}{2}} \| M_T -\ept M_T \|_{L_p(\P)} \mer <  +\infty$.
      \end{enumerate}
      One can easily check that  (i$_{\theta}) \Longleftrightarrow$ (i$'_{\theta}$) for $\theta\in (0,1]$
      and $q\in [1,\infty]$:
      Indeed, for any random variables $U$ and $V$, respectively bounded and in $L_p(\P)$, observe that
      \equa
      &    & \|U V-\ept(UV)\|_{L_p(\P)} \\
      &\leq& \big\|[U-\ept U]V\big\|_{L_p(\P)}+\big\|\ept(U) [V-\ept V] \big \|_{L_p(\P)} \\
      &    & \hspace{10em} + \big \|\ept(U [\ept(V)-V]) \big \|_{L_p(\P)} \\
      &\leq& \|[U-\ept U]V\|_{L_p(\P)}+2\|U\|_\infty\| V-\ept V\|_{L_p(\P)}.
      \tion
      For $U:=e^{\int_0^T k(r,X_r) dr}$ and $V:=g(X_T)$ we have 
      \[ |U-\ept U|\leq 2  \|k\|_\infty (T-t)e^{\|k\|_\infty T} \]
      and obtain
      \equa
      &   & \|e^{\int_0^T k(r,X_r) dr} g(X_T)-\ept(e^{\int_0^T k(r,X_r) dr} g(X_T))\|_{L_p(\P)} \\
      &\le& 2 e^{\|k\|_\infty T}  \Big [\|k\|_\infty(T-t)\|g(X_T)\|_{L_p(\P)}+\| g(X_T)-\ept g(X_T)\|_{L_p(\P)}
            \Big ].
      \tion
      This proves (i$_{\theta}$)$\Longrightarrow$ (i$'_{\theta}$). The converse is proved similarly
      by letting $U:= e^{-\int_0^T k(r,X_r) dr}$ and 
      $V:=e^{\int_0^T k(r,X_r) dr} g(X_T)$. 
\item \underline{The case $\theta=1$ and $q=\infty$}. 
      \begin{enumerate}
      \item One has 
      $({\rm i}_1') \Longleftrightarrow ({\rm ii}_1) 
                   \Longrightarrow     ({\rm iii}_1)$:
             First we observe that 
             \begin{equation}\label{eqn:sqrt-1}
             \Phi_\infty \Big( (T-t)^{-\frac{1}{2}} \left ( \int_t^T h(s)^2 ds \right )^\frac{1}{2} 
             \Big) 
             \le \Phi_\infty (h).
             \end{equation}
            
             Then (ii$_1$)$\Longrightarrow$ (i$_1'$)
             follows from (\ref{eqn:sqrt-1}) with 
             $h(t)= \| \nabla v (t,X_t)\|_{L_p(\P)}$ and 
             Lemma \ref{lemma:upper_bound_diffenerence_conditional_expectations}.
             The implications (i$_1'$)$\Longrightarrow$ (ii$_1$)
             and 
             (i$_1$)$\Longrightarrow$ (iii$_1$)
             follow by Lemmas \ref{lemma:H1} and \ref{lemma:H2_new}.
       \item The implication 
             $({\rm iii}_{1}) \Longrightarrow ({\rm ii}_{1})$ is not true in general.
             Take 
             $p=2$, $q=\infty$, $X=B$, $\P=\Q$, $T=1$, $k=0$ and $d=1$,
             then the counterexample $g(x)=\sqrt{x\vee 0}$ is discussed in \cite{geis:geis:gobe:12}.
       \end{enumerate}

\item Let $(\Omega,\cF,(\cF_t)_{t\in [0,T]},\P)$ be a stochastic basis satisfying the usual conditions, i.e.
      $(\Omega,\cF,\P)$ is complete,
      $(\cF_t)_{t\in [0,T]}$ is right-continuous, 
      $\cF_0$ is generated by the null-sets of $\cF$ and where we can assume w.l.o.g. that
      $\cF=\cF_T$. Assume further that the filtration is obtained as augmentation of the natural
      filtration of a standard $d$-dimensional Brownian motion $W=(W_t)_{t\in [0,T]}$ starting in zero.
      It is known \cite[Corollary 1 on p. 187]{prot:04} that on this stochastic basis all local martingales are             
      continuous.
      Assume a progressively measurable $d$-dimensional process $\beta=(\beta_t)_{t\in [0,T]}$ 
      with $\sup_{t,\omega}|\beta_t(\omega)|<\infty$ and consider the unique strong solution of the
      SDE
      \[ X_t = x_0 + \int_0^t \sigma(s,X_s) dW_s + \int_0^t b(s,X_s) ds -
                     \int_0^t \beta_s ds. \]
      Letting, 
      \equa
      \gamma_s    &:= & \sigma^{-1} (s,X_s)\beta_s, \\
      B_t         &:= & W_t - \int_0^t \gamma_s ds, \\
      1/\lambda_t &:= & e^{ \int_0^t \gamma_s^\top dW_s - \frac{1}{2} \int_0^t |\gamma_s|^2 ds}
                    =   e^{ \int_0^t \gamma_s^\top dB_s + \frac{1}{2} \int_0^t |\gamma_s|^2 ds},\\
      d\Q     &:= & (1/\lambda_T) d\P,
      \tion
      we obtain by the Girsanov Theorem that
      $(\Omega,\cF,(\cF_t)_{t\in [0,T]},\Q)$, $(B_t)_{t\in [0,T]}$ and $(X_t)_{t\in [0,T]}$ satisfy the 
      assumptions of our paper (i.e. all  martingales are continuous
      - which can be checked by expressing the conditional expectation under 
      $\Q$ by the  conditional expectation under $\P$-, so that local martingales 
      are continuous as well) and that
      $\lambda_T \in A_\alpha$ for all $\alpha\in (1,\infty)$.
      Hence the passage from $\Q$ to $\P$ corresponds to adding a drift to the diffusion $X$.

\item In the case the drift term in item (7) is Markovian, i.e. $\beta_t=\beta(t,X_t)$ for an
      appropriate $\beta:[0,T]\times\R^d \to \R^d$, and if we let
      $Y_t:= v(t,X_t)$ and $Z_t:= \nabla v(t,X_t)\sigma(t,X_t)$, then we get the BSDE
      \equa
      -dY_t & = & [k(t,X_t)Y_t +Z_t\sigma^{-1}(t,X_t)\beta_t]dt-Z_t dW_t,\\
         Y_T& = & g(X_T).
      \tion
      Then  it is proved in \cite{geis:geis:gobe:12} under certain conditions the equivalence between the 
      following assertions for $p\in [2,\infty)$, $\theta\in(0,1]$ and polynomially bounded $g$:
      \begin{enumerate}
      \item $\sup_{t\in [0,T)}(T-t)^{-\frac{\theta}{2}}  \| g(X_T) - \E^{\cF_t}(g(X_T))\|_{L_p(\P)} < +\infty$.
      \item $\sup_{t\in [0,T)}(T-t)^{\frac{1-\theta}{2}} \| Z_t\|_{L_p(\P)} < +\infty$.
      \end{enumerate}
      These are the analogues of (i$_{\theta}$) and (ii$_{\theta}$) for $q=\infty$. 
\end{enumerate}
\end{remark}


\section{Proof of Theorem \ref{thm:equivalence}}
Through the whole section we assume that the condition (P) is satisfied.

\subsection{Preliminaries}

To estimate $L_p$ norms under different measures, the following lemma is useful.

\begin{lemma}\label{lemma:1}
For any $1<\alpha<p<\infty$, $\lambda_T\in A_{\alpha}(\Q)$, $r:=\frac{p}{p-\alpha}$,
$U\in L_p(\Omega,\cF,\P)$, $V\in L_r(\Omega,\cF,\Q)$ and
$c_{\eqref{eqn:lemma:1}}>0$ such that 
$[\eqt (  | \frac{\la_t}{\la_T}  |^{\frac{1}{\alpha-1}} ) ]^{\frac{\alpha-1}{p}} \le c_{\eqref{eqn:lemma:1}}$ a.s.
we have that
\begin{equation}\label{eqn:lemma:1}    \eqt|UV|
   \le c_{\eqref{eqn:lemma:1}}
                  \left [ \ept |U|^p \right ]^{\frac{1}{p}} [ \eqt   | V |^{r}                   ]^{\frac{1}{r}} \mbox{ a.s.}
                  \end{equation}
\end{lemma}     

\begin{proof} 
Letting $1=\frac{1}{p}+\frac{1}{p'} = \frac{1}{\alpha}+\frac{1}{\alpha'}$
one has a.s. that 
\equa
      \eqt | U V | 
& = & \lambda_t \ept (|UV|/\lambda_T ) \\
&\le& \lambda_t [\ept |U|^p]^\frac{1}{p}
                [\ept (|V|^{p'} \lambda_T^{-\frac{p'}{r}} \lambda_T^{-p'+\frac{p'}{r}} )]^\frac{1}{p'} \\
&\le& \lambda_t [\ept |U|^p]^\frac{1}{p}
                [\ept (|V|^r /\lambda_T)]^\frac{1}{r}
      [\ept \lambda_T^{-\alpha'}]^{\frac{r-p'}{p'r}} \\
&\le& c_{\eqref{eqn:lemma:1}}
                [\ept |U|^p]^\frac{1}{p}
                [\eqt |V|^r]^\frac{1}{r}.
\tion
\end{proof}
As simple consequences of this lemma for $V\equiv 1$ , observe that 
\begin{equation}\label{eq:lemma1:bis}    
       \|\eqt U\|_{L_p(\P)}
   \le c_{\eqref{eqn:lemma:1}} \|U\|_{L_p(\P)}
   \sptext{1}{for}{1} U\in L_p(\P).
\end{equation}
In the next step we will estimate $\nabla v(t,X_t)$ and $D^2v(t,X_t)$ in Lemmas \ref{lemma:H1} and
\ref{lemma:H2_new} from above by conditional moments of $M_T=K^X_T g(X_T)$ and $g(X_T)$,
and extend therefore Lemma \ref{lemma:Malliavin_weights} to the case $k=0$ and allow a change of
measure by Muckenhoupt weights.

\begin{lemma}\label{lemma:H1} 
For any $p\in (1,\infty)$, we have a.s. that
\begin{equation}
      |\nabla v(t,X_t)|
 \le c_{\eqref{eqn:lemma:H1}}\left [
        \frac{\Big( \eqt  |M_T -\eqt M_T |^p \Big)^{\frac{1}{p}}}
             {\sqrt{T-t}}
           +(T-t)\Big(\eqt |M_T|^p\Big)^{\frac{1}{p}}\right ] ,
      \label{eqn:lemma:H1}
\end{equation}
where $c_{\eqref{eqn:lemma:H1}}>0$ depends at most on $(\sigma,b,k,p)$.
The same estimate holds true
if the measure $\Q$ is replaced by the measure $\P$ with
$\lambda\in A_{\alpha}(\Q)$ and $\alpha \in (1,p)$,
where the constant $c_{\eqref{eqn:lemma:H1}}>0$ might additionally depend on $\Q$
(and therefore implicitly on $\alpha$). 
\end{lemma}     
\smallskip

\begin{proof}
The statement for $\P$ for $p\in (1,\infty)$ can be deduced from the statement for 
$\Q$ for $q \in (1,p)$. Let us fix $1<q<p<\infty$, define $p_0:=p/q\in (1,\infty)$,
take $r\in (p_0',\infty)$ and let $\beta:= \frac{p_0'r-p_0'}{r-p_0'}$. For $\lambda\in A_\alpha(\Q)$
with $1=(1/\alpha)+(1/\beta)$ we apply Lemma \ref{lemma:1} with $p$ replaced by $p_0$ 
and get 
\[     \left ( \E^{{\cF_t}}_\Q |Z|^q \right )^\frac{1}{q} 
   \le c_{\eqref{eqn:lemma:1}}^\frac{1}{q} \left ( \E^{{\cF_t}}_\P |Z|^{p} \right )^\frac{1}{p} \]
and, by \eqref{eq:lp},
\[     \left ( \E^{{\cF_t}}_\Q |Z-\E^{{\cF_t}}_\Q Z|^q \right )^\frac{1}{q} 
   \le 2 \left ( \E^{{\cF_t}}_\Q |Z-\E^{{\cF_t}}_\P Z|^q \right )^\frac{1}{q} 
   \le 2 c_{\eqref{eqn:lemma:1}}^\frac{1}{q} \left ( \E^{{\cF_t}}_\P |Z-\E^{{\cF_t}}_\P Z|^{p} 
       \right )^\frac{1}{p}\]
whenever $Z\in \bigcap_{r\in [1,\infty)} L_r(\Q)$ (cf. Remark \ref{remark:L_p-inclusions}). 
Because 
$   \lim_{r\to\infty} \frac{p_0'r-p_0'}{r-p_0'}
  = p_0' = \frac{p}{p-q}$ and the convergence is from above, we can take $\beta$ to be in 
 $\kla \frac{p}{p-q},\infty \mer$. Sending $q$ to $1$ gives that
 $\beta \in \kla \frac{p}{p-1},\infty \mer$ or $\alpha\in (1,p)$.
\bigskip

Now we follow a martingale approach (see, for example, \cite{gobe:muno:05}) and prove the statement for
the measure $\Q$. 
\bigskip

(a) We define $(\nabla X_t)_{t\in [0,T]}$ to be the solution of a linear SDE 
(see \cite[Chapter 5]{prot:04}):
\[   \nabla X_t  
   = I_d + \sum_{j=1}^d\int_0^t \nabla\sigma_j(s,X_s) \nabla X_s dB^j_s 
         +\int_0^t  \nabla b(s,X_s) \nabla X_s ds \]
         and $\sigma(.)=(\sigma_1(.),\dots,\sigma_d(.)).$ This matrix-valued process is a.s. invertible and its inverse satisfies
\equa
[\nabla X_t]^{-1}
& = & I_d- \sum_{j=1}^d \int_0^t [\nabla X_s]^{-1} \nabla\sigma_j(s,X_s) dB^j_{s}\\
&   & -\int_0^t [\nabla X_s]^{-1}(\nabla b(s,X_s) -\sum_{j=1}^d (\nabla\sigma_j(s,X_s))^2) ds.
\tion
(b) Next we show that $(N_t)_{t\in [0,T)}$ with
\[
       N_t 
   :=   K^X_t\nabla v(t,X_t)\nabla X_t
      + \left (\int_0^t \nabla k(s,X_s) \nabla X_s ds\right )  M_t 
\]
is a $\Q$-martingale. One way consists in using It\^o's formula to
verify that $N$ is a martingale. In fact,
the bounded variation term in the It\^o-process decomposition
of $N$ is 
\[ \int_0^t \left [ K_s^X k(s,X_s) \nabla v(s,X_s)\nabla X_s + K_s^X {C_s} \right ] ds 
   + \int_0^t \left [ \nabla k(s,X_s) \nabla X_s M_s \right ] ds, \]
where $\int_0^t  {C_s} ds$ is the bounded variation term of $\nabla v(t,X_t)\nabla X_t$. 
Hence it is sufficient to show that
\[  {C_s} = -  \nabla [ v(s,X_s) k(s,X_s) ] \nabla X_s.\]
The PDE for $w=\nabla v$ on $[0,T)\times\R^d$ reads as
\begin{equation}\label{eqn:pde_for_gradient}
       \frac{\partial}{\partial t} w_i + \frac{1}{2} \langle A,D^2 w_i \rangle 
                                   + \langle b, (\nabla w_i)^T \rangle
    = - \frac{1}{2} \langle  \partial_{x_i} A,D^2 v \rangle - \langle \partial_{x_i} b, w^T \rangle
     - \partial_{x_i} ( v k ).
\end{equation}
By a simple computation this gives that the bounded variation term of
$(\sum_{i=1}^d \frac{\partial v}{\partial x_i}(t,X_t) (\nabla X_t)_{il})_{t\in [0,T)}$ 
computes as
$ - \sum_{i=1}^d \frac{\partial (v k)}{\partial x_i}(t,X_t) (\nabla X_s)_{il} dt$
and step (b) is complete.
\smallskip

(c)  Exploiting the martingale property of $N$ between $t$ and some deterministic 
$S\in(t,T)$, we have
\begin{eqnarray}
&   & (S-t)\left [K^X_t\nabla v(t,X_t)\nabla X_t+\left (\int_0^t \nabla k(s,X_s) \nabla    
      X_s ds \right )  M_t \right ] \nonumber \\
& = & \eqt \bigg (\int_t^S \big [K^X_r\nabla v(r,X_r)\nabla X_r \nonumber
      + \left (\int_0^r \nabla k(s,X_s) \nabla X_s ds \right )  M_r \big ]
        dr \bigg ) \nonumber \\
& = & \eqt \left ( \left [\int_t^{S} K^X_r\nabla v(r,X_r) \sigma(r,X_r) dB_r \right ]
                  \left [\int_t^{S}  (\sigma(r,X_r)^{-1} \nabla X_r)^\top dB_r 
                                            \right ]^\top \right ) \nonumber \\
&   & + (S-t) M_t \left [ \int_0^t \nabla k(s,X_s) \nabla X_s ds\right ] 
      \nonumber \\
&   & + \eqt \left ( M_S \int_t^S \left [\int_t^r \nabla k(s,X_s) \nabla X_s ds\right ] dr\right ).
      \label{eq:tmp:1}
\end{eqnarray}
At the last equality, we have used the $\Q$-martingale property of $(M_t)_{t\in [0,T]}$
and the conditional It\^o isometry 
\[ \eqt \left (\left [\int_t^S A_{1,r} dB_r\right ] \left [\int_t^S A_{2,r} dB_r\right ]^\top\right ) =\eqt\left (\int_t^S A_{1,r} A^\top_{2,r} dr\right ) \]
(available for any square integrable and progressively measurable matrix-valued processes $(A_{1,r})_r$ and $(A_{2,r})_r$, having $d$ columns  and an arbitrary number of rows).
After simplifications, \eqref{eq:tmp:1} writes
\equa
&   & (S-t)K^X_t\nabla v(t,X_t)\nabla X_t \\
& = & \eqt \left ([M_S-M_t]\left [\int_t^{S}  (\sigma(r,X_r)^{-1} \nabla X_r)^\top dB_r\right ]^\top\right )\\
&   & +\eqt\left (M_S \left [\int_t^S (S-s) \nabla k(s,X_s) \nabla X_s ds\right ] \right ).
\tion
Using that $M_S\to M_T$ in $L_2(\Q)$ we derive
\equa
&   & (T-t)K^X_t\nabla v(t,X_t) \\
& = & \eqt\left ([ M_T - M_t ]\left [\int_t^{T}  (\sigma(r,X_r)^{-1} \nabla X_r[\nabla X_t]^{-1})^\top dB_r
      \right ]^\top \right ) \\
&   & +\eqt\left ( M_T \left [ \int_t^T (T-s) \nabla k(s,X_s) \nabla X_s[\nabla X_t]^{-1} ds \right ] \right ).
\tion
Finally, observe that $\sup_{t\in [0,T)}\sup_{r\in [t,T]}\eqt(|\nabla X_r[\nabla X_t]^{-1}|^q)$ is a bounded random variable for any $q\geq 1$; therefore, standard computations  complete our assertion.
\end{proof}
\medskip

For the following we let
$m(t,x)   :=  v(t,x) k(t,x)$.
\begin{lemma}\label{lemma:m_M}
For $0\le r <t \le T$ and $1<p_0<p<\infty$
one has a.s. that
\begin{multline}\label{eqn:lemma:m_M}
     \left ( \E^{\cF_r}_\Q | m(t,X_t) - \E^{\cF_r}_\Q m(t,X_t) |^{p_0} \right )^\frac{1}{p_0} \\
 \le c_{\eqref{eqn:lemma:m_M}} \left [ \sqrt{t-r} \left ( \E^{\cF_r}_\Q |M^*|^p \right )^\frac{1}{p}
                + \left ( \E^{\cF_r}_\Q | M_t - M_r |^{p_0}
                  \right )^\frac{1}{p_0} \right ] 
\end{multline}
where $M^*:= \sup_{s\in [0,T]} |M_s|$ and $c_{\eqref{eqn:lemma:m_M}}>0$ depends at most on
$(p_0,p,\sigma,b,k)$.
\end{lemma}

\begin{proof}
(a) For 
$\frac{1}{p_0} = \frac{1}{q_k}+\frac{1}{r_k}
             = \frac{1}{s_k}+\frac{1}{t_k}+\frac{1}{r_k}$
with $r_k,s_k,t_k\in {[p_0,\infty]}$,             
a sub-$\sigma$-algebra $\mathcal{G}\subseteq \cF$, 
$\underline{U}_k:=U_1\cdots U_k$ and
$\overline{U}_k:=U_k\cdots U_N$ with 
$\underline{U}_0:=1$ and  
$\overline{U}_{N+1}:=1$, and for 
 $\underline{U}_{k-1}\in L_{t_k}(\Q)$,
 $U_k \in L_{s_k}(\Q)$,
 $\overline{U}_{k+1}\in L_{r_k}(\Q)$, where $k=1,...,N$,
we get by a telescoping sum argument and the conditional H\"older inequality  that 
\equa
&   & \kla \E^\mathcal{G}_\Q |U_1\cdots U_N - E^\mathcal{G}_\Q(U_1 \cdots U_N)|^{p_0} \mer^\frac{1}{p_0} \\
&\le& \sum_{k=1}^N \kla \E^\mathcal{G}_\Q |[\E^\mathcal{G}_\Q (\underline{U}_{k-1})] U_k  - 
                              \E^\mathcal{G}_\Q (\underline{U}_k) |^{q_k} \mer^\frac{1}{q_k} 
      \kla \E^\mathcal{G}_\Q |\overline{U}_{k+1}|^{r_k} \mer^\frac{1}{r_k} \\
&\le& \sum_{k=1}^N \kla \E^\mathcal{G}_\Q |[\E^\mathcal{G}_\Q (\underline{U}_{k-1})] U_k  - 
              \E^\mathcal{G}_\Q (\underline{U}_{k-1}) \E^\mathcal{G}_\Q U_k 
                              |^{q_k} \mer^\frac{1}{q_k} 
       \kla \E^\mathcal{G}_\Q |\overline{U}_{k+1}|^{r_k} \mer^\frac{1}{r_k} \\     
&   & +  \sum_{k=1}^N \kla \E^\mathcal{G}_\Q |[\E^\mathcal{G}_\Q (\underline{U}_{k-1})] \E^\mathcal{G}_\Q U_k  - 
                              [\E^\mathcal{G}_\Q (\underline{U}_{k  })] |^{q_k} \mer^\frac{1}{q_k} 
      \kla \E^\mathcal{G}_\Q |\overline{U}_{k+1}|^{r_k} \mer^\frac{1}{r_k}      \\
&\le& 2 \sum_{k=1}^N 
      \kla \E^\mathcal{G}_\Q | \underline{U}_{k-1}|^{t_k} \mer^\frac{1}{t_k}
      \kla \E^\mathcal{G}_\Q |U_k  - \E^\mathcal{G}_\Q U_k|^{s_k} \mer^\frac{1}{s_k} 
      \kla \E^\mathcal{G}_\Q |\overline{U}_{k+1}|^{r_k} \mer^\frac{1}{r_k}.
\tion 
(b) We apply (a) to $N=3$ and $m(s,X_s) =  k(s,X_s) (K_s^X)^{-1} M_s$ to derive
\equa
&   &  \left ( \E^{\cF_r}_\Q | m(t,X_t) - \E^{\cF_r}_\Q m(t,X_t) |^{p_0} \right )^\frac{1}{p_0} \\
&\le&  2 \|k\|_\infty e^{T\|k\|_\infty} \left ( \E^{\cF_r}_\Q | M_t - M_r |^{p_0} \right )^\frac{1}{p_0}  \\
&   & + 2  \left ( \E^{\cF_r}_\Q | k(t,X_t) - \E^{\cF_r}_\Q k(t,X_t) |^\beta \right )^\frac{1}{\beta}
      e^{T\|k\|_\infty}  \left ( \E^{\cF_r}_\Q |M^*|^p \right )^\frac{1}{p} \\
&   &  + 2 \|k\|_\infty \left ( \E^{\cF_r}_\Q | (K_t^X)^{-1} - \E^{\cF_r}_\Q (K_t^X)^{-1} |^\beta 
      \right )^\frac{1}{\beta}  \left ( \E^{\cF_r}_\Q |M^*|^p \right )^\frac{1}{p}
\tion
for $\frac{1}{p_0}=\frac{1}{p} + \frac{1}{\beta}$.
We conclude by
\equa
      \left ( \E^{\cF_r}_\Q | k(t,X_t)  - \E^{\cF_r}_\Q k(t,X_t) |^{\beta} \right )^\frac{1}{\beta} 
&\le&  2  \left ( \E^{\cF_r}_\Q | k(t,X_t)  - k(t,X_r) |^{\beta} \right )^\frac{1}{\beta} \\
&\le&  2     \| \nabla k\|_\infty 
        \left ( \E^{\cF_r}_\Q | X_t  - X_r |^{\beta} \right )^\frac{1}{\beta} \\
&\le&  2 \| \nabla k\|_\infty  c(b,\sigma,\beta) \sqrt{t-r}  
\tion
and 
$\left ( \E^{\cF_r}_\Q | (K_t^X)^{-1}  - \E^{\cF_r}_\Q (K_t^X)^{-1} |^{\beta} 
      \right )^\frac{1}{\beta}
        \le 2 \|k \|_\infty (t-r) e^{T \|k\|_\infty}$.
\end{proof}
\medskip

\begin{lemma}\label{lemma:proportional_estimate}
For $0\le r < t < T$ and $p\in (1,\infty)$ one has a.s. that
\begin{equation}\label{eqn:lemma:proportional_estimate}
     \left ( \E^{\cF_r}_\Q | M_t - M_r |^p \right )^\frac{1}{p}
   \le c_{\eqref{eqn:lemma:proportional_estimate}} \left [ 
           \left ( \frac{t-r}{T-t} \right )^\frac{1}{2}
             \left ( \E^{\cF_r}_\Q | M_T - M_r |^p \right )^\frac{1}{p}
           + (t-r)^\frac{1}{2} |M_r| 
         \right ]
\end{equation}
where $c_{\eqref{eqn:lemma:proportional_estimate}}\ge 1$ depends at most on $(p,\sigma,b,k)$.
\end{lemma}

\begin{proof}
Let $p_0:=\frac{1+p}{2}$,
$\lambda_u := K_u^X \nabla v(u,X_u) \sigma(u,X_u)$
and $0\le r \le u \le t$. Then Lemma \ref{lemma:H1} implies that
\equa
&   & |\lambda_u| {e^{-T \|k\|_\infty}}\\
&\le& \|\sigma\|_\infty c_{\eqref{eqn:lemma:H1},p_0}
      \bigg [ (T-u)^{-\frac{1}{2}}
               \Big( \equ  |M_T-M_u|^{p_0}\Big)^{\frac{1}{p_0}} \\
&   &  \hspace*{18em} + (T-u)\Big(\equ |M_T|^{p_0}\Big)^{\frac{1}{p_0}}\bigg ] \\
&\le& \|\sigma\|_\infty c_{\eqref{eqn:lemma:H1},p_0}
      \bigg [ (T-u)^{-\frac{1}{2}}
             2 \Big( \equ  |M_T-M_r|^{p_0}\Big)^{\frac{1}{p_0}} \\
&   &  \hspace*{7em}
      + (T-u)\Big(\equ |M_T-M_r|^{p_0}\Big)^{\frac{1}{p_0}}
      + (T-u) |M_r|
          \bigg ] \\
&\le& \|\sigma\|_\infty c_{\eqref{eqn:lemma:H1},p_0}
      \bigg [ [2 + T^\frac{3}{2}](T-u)^{-\frac{1}{2}} 
               \Big( \equ  |M_T-M_r|^{p_0}\Big)^{\frac{1}{p_0}}
     + (T-u)|M_r|\bigg ] \\
&\le& \|\sigma\|_\infty c_{\eqref{eqn:lemma:H1},p_0}
      [2 + T^\frac{3}{2}+T] \bigg [ (T-t)^{-\frac{1}{2}} 
               \Big( \equ  |M_T-M_r|^{p_0}\Big)^{\frac{1}{p_0}}
     + |M_r|\bigg ].  
\tion
Letting $c:= {e^{T \|k\|_\infty}}\|\sigma\|_\infty c_{\eqref{eqn:lemma:H1},p_0}
         [2 + T^\frac{3}{2}+T]$ we conclude the proof by using the Burkholder-Davis-Gundy inequalities
in order to get
\equa
&    & \frac{1}{a_p}\left ( \E^{\cF_r}_\Q | M_t - M_r |^p \right )^\frac{1}{p} \\
&\le& \left ( \E^{\cF_r}_\Q \kla \int_r^t |\lambda_u|^2 du \mer^\frac{p}{2} \right )^\frac{1}{p} \\
&\le& c \bigg [  
            (T-t)^{-\frac{1}{2}} 
            \kla \E^{\cF_r}_\Q \kla \int_r^t  \Big( \equ  |M_T-M_r|^{p_0}\Big)^{\frac{2}{p_0}} du \mer^\frac{p}{2} 
             \mer^\frac{1}{p} \\
&   & \hspace*{22em}                  + \sqrt{t-r} |M_r| \bigg ] \\
&\le& c \left [  
            \sqrt{\frac{t-r}{T-t}}
            \kla \E^{\cF_r}_\Q \Big (  \sup_{u\in [r,t]} \equ  |M_T-M_r|^{p_0}\Big)^{\frac{p}{p_0}}  
             \mer^\frac{1}{p} 
                    + \sqrt{t-r} |M_r| \right ] \\
&\le& c \bigg [  \kla \frac{p/p_0}{(p/p_0)-1}\mer^\frac{1}{p_0}
            \sqrt{\frac{t-r}{T-t}}
            \kla \E^{\cF_r}_\Q \Big (  \eqt  |M_T-M_r|^{p_0}\Big)^{\frac{p}{p_0}}  
             \mer^\frac{1}{p} \\
&   & \hspace*{22em}  + \sqrt{t-r} |M_r| \bigg ] \\
&\le& c \left [  \kla \frac{p}{p-p_0}\mer^\frac{1}{p_0}
            \sqrt{\frac{t-r}{T-t}}
            \kla \E^{\cF_r}_\Q  \eqt  |M_T-M_r|^p  
             \mer^\frac{1}{p} 
                    + \sqrt{t-r} |M_r| \right ].
\tion
\end{proof}

\begin{lemma}\label{lemma:H2_new}
For $p\in (1,\infty)$ there is a constant $c_\eqref{eqn:lemma:H2_new}=c(\sigma,b,k,p)>0$
such that one has a.s. that
\begin{equation}\label{eqn:lemma:H2_new}
         |D^2 v(r,X_r)| \le \\ c_\eqref{eqn:lemma:H2_new} \bigg [ 
     \frac{\kla \E_\Q^{\cF_r} \bet g(X_T) - \E_\Q^{\cF_r}g(X_T)\rag^p \mer^\frac{1}{p}}
          {T-r}
        +  \sqrt{T-r}
      \left ( \E^{\cF_r}_\Q |M^*|^{p} \right )^\frac{1}{p}
      \bigg ].
\end{equation}
The same estimate holds true
if the measure $\Q$ is replaced by the measure
$\P$ with $\lambda\in A_{\alpha}(\Q)$ and $\alpha \in (1,p)$, where the 
constant $c_\eqref{eqn:lemma:H2_new}>0$ might additionally depend on $\Q$ (and therefore implicitly 
on $\alpha$).
\end{lemma}

\begin{proof}
(a) The statement for $\P$ for $p\in (1,\infty)$ can be deduced from the statement for 
$\Q$ for $q \in (1,p)$ as in the first step of the proof of Lemma \ref{lemma:H1}.
\smallskip

(b) Now we show the estimate for the measure $\Q$.
For $0\le s \le t \le T$, a fixed $T_0\in (0,T)$ and $r\in [0,T_0]$
we let
\[ v^t(s,x) :=  \E_\Q \kla m(t,X_t) | X_s=x   \mer
   \sptext{.5}{and}{.5}
   v_h(r,x) :=   \E_\Q \kla v(T_0,X_{T_0})   | X_r = x \mer. \]
It\^o's formula applied to $v$ gives for $r\in [0,T_0]$ that
\[ v (r,x) = \E_\Q  \kla  v(T_0,X_{T_0})
                + \int_r^{T_0} (kv)(t,X_t) dt | X_r = x \mer \]
and therefore
\[ v(r,x)  =  v_h(r,x) + \int_r^{T_0} v^t(r,x) dt. \]
Using Lemma \ref{lemma:Malliavin_weights} and the arguments from Remark \ref{remark:thm_equivalence}(5)
one can show for $0\le r < t \le T_0<T$ that
\begin{equation}\label{eqn:estimates_v^t}
   |\nabla v^t(r,x)| \le \gamma e^{\gamma |x|^{k_g}}
   \sptext{1}{and}{1}
   |D^2 v^t(r,x)| \le \frac{\gamma}{\sqrt{t-r}} e^{\gamma |x|^{k_g}},
\end{equation}
where $\gamma>0$ depends at most on $(\sigma,b,k,K_g,k_g,T_0)$. From this
we deduce that 
\[
 D^2 v(r,x)  =  D^2 v_h(r,x) + \int_r^{T_0} D^2 v^t(r,x) dt
\]
where \eqref{eqn:estimates_v^t} are used to interchange the integral and $D^2$.
For $p_0:= \frac{1+p}{2}$, $0\le r < t \le T$ and $s\in [0,T_0)$ we again use
Lemma \ref{lemma:Malliavin_weights} to get
\equa
 |D^2 v^t (r,X_r)| &\le& \frac{\kappa_{p_0}}{(t-r)} 
                          \kla \E_\Q^{\cF_r} \bet m(t,X_t) - \E_\Q^{\cF_r}{m}(t,X_t) 
                          \rag^{p_0}
                          \mer^\frac{1}{p_0} \quad \mbox{a.s.}, \\
|D^2 v_h (s,X_s)| &\le& \frac{\kappa_p}{(T_0-s)} 
                          \kla \E_\Q^{\cF_s} \bet v(T_0,X_{T_0}) - \E_\Q^{\cF_s} v(T_0,X_{T_0}) 
                          \rag^p \mer^\frac{1}{p}\quad \mbox{a.s.}
\tion
From the first estimate we derive by Lemmas \ref{lemma:m_M} and \ref{lemma:proportional_estimate}
(with $p$ replaced by $p_0$) a.s. that
\equa
&   & |D^2 v^t (r,X_r)| \\
&\le& \frac{\kappa_{p_0}}{(t-r)} 
                          \kla \E_\Q^{\cF_r} \bet m(t,X_t) - \E_\Q^{\cF_r}{m}(t,X_t) 
                          \rag^{p_0}
                          \mer^\frac{1}{p_0} \\
&\le& \frac{\kappa_{p_0} c_{\eqref{eqn:lemma:m_M}}}{(t-r)}             
      \left [ \sqrt{t-r} \left ( \E^{\cF_r}_\Q |M^*|^p \right )^\frac{1}{p}
                + \left ( \E^{\cF_r}_\Q | M_t - M_r |^{p_0} 
                  \right )^\frac{1}{p_0} \right ]  \\
&\le& \kappa_{p_0} c_{(\ref{eqn:lemma:m_M})}[ 1  + c_{(\ref{eqn:lemma:proportional_estimate})}
         ]   \frac{1}{\sqrt{t-r}}
      \left ( \E^{\cF_r}_\Q |M^*|^p \right )^\frac{1}{p} \\
&   & +  \kappa_{p_0}  c_{(\ref{eqn:lemma:m_M})}  c_{(\ref{eqn:lemma:proportional_estimate})}
        \frac{1}{\sqrt{T-t}\sqrt{t-r}}
         \left ( \E^{\cF_r}_\Q | M_T - M_r |^{p_0} \right )^\frac{1}{p_0}
 \tion
and
\[     \int_r^T |D^2 v^t(r,X_r)| dt
   \le   c \left [ \sqrt{T-r}
      \left ( \E^{\cF_r}_\Q |M^*|^p \right )^\frac{1}{p} 
  + 
         \left ( \E^{\cF_r}_\Q | M_T - M_r |^p \right )^\frac{1}{p}\right ] \]
with $c:= \kappa_{p_0} c_{(\ref{eqn:lemma:m_M})} 
          \max \{ 2 + 2 c_{(\ref{eqn:lemma:proportional_estimate})}, 
                  c_{(\ref{eqn:lemma:proportional_estimate})}
                  {\mathrm{Beta}}(\frac{1}{2},\frac{1}{2}) \}$. 
The second estimate yields by $T_0\uparrow T$ and (\ref{eqn:L_p-convergence_to_maturity}) that
\[ |D^2 v_h (r,X_r)| \le \frac{\kappa_p}{(T-r)} 
                          \kla \E_\Q^{\cF_r} \bet g(X_T)- \E_\Q^{\cF_r} g(X_T)
                          \rag^p \mer^\frac{1}{p} \]
and the upper bound is independent of $T_0$. Combining the estimates with
 \begin{multline*}
    \kla \E^{\cF_r}_\Q | M_T - M_r |^p \right )^\frac{1}{p} 
    \le 2 e^{\|k\|_\infty T} \\
         \left [ \| k\|_\infty (T-r) e^{\|k\|_\infty T} \left ( \E^{\cF_r}_\Q |M^*|^p \right )^\frac{1}{p} +   \kla \E_\Q^{\cF_r} \bet g(X_T) - \E_\Q^{\cF_r} g(X_T) \rag^p \mer^\frac{1}{p} \right ]
\end{multline*}
using the arguments from Remark \ref{remark:thm_equivalence}(5) the proof is complete.
\end{proof}
\medskip

\begin{lemma}\label{lemma:upper_bound_diffenerence_conditional_expectations}
Let $\lambda = \mathcal{E}(Y)$, where $Y$ is a {\rm BMO}-martingale with $Y_0=0$.
Then, for $p\in (1,\infty)$, $t\in[0,T]$ and  
$c_\eqref{eqn:lemma:upper_bound_diffenerence_conditional_expectations}:= 2 b_p(\P)e^{T\|k\|_\infty}\|\sigma\|_\infty $
we have that
\begin{equation}
\label{eqn:lemma:upper_bound_diffenerence_conditional_expectations}
      \| M_T - \ept M_T \|_{L_p(\P)}
   \le c_\eqref{eqn:lemma:upper_bound_diffenerence_conditional_expectations} 
       \noo \Big( \int_t^T |\nabla v(s,X_s)|^2ds  \Big)^\frac{1}{2} \rrm_{L_p(\P)}\!.
\end{equation}
\end{lemma}

\begin{proof}
Owing to inequality \eqref{eq:lp} and applying Proposition \ref{proposition:BMO-exponential}, we get
\equa
       \| M_T - \ept M_T \|_{L_p(\P)} 
&\le& 2 \| M_T - M_t \|_{L_p(\P)} \\
&\le& 2 b_p(\P) \left \| \sqrt{\langle M \rangle_T - \langle M \rangle_t} \right \|_{L_p(\P)} \\
& = & 2 b_p(\P) \noo\sqrt{\int_t^T |K^X_s\nabla v(s,X_s) \sigma(s,X_s)|^2ds} \rrm_{L_p(\P)}.
\tion
\end{proof}
\medskip

\begin{lemma}\label{lemma:upper_bound_nabla}
For $p\in [2,\infty)$, $\lambda_T \in A_\alpha(\Q)$ with $\alpha\in (1,p)$,
$0\le s < t < T$ and $l=1,...,d$ we have that 
\begin{align}
\nonumber   &   \noo K^X_t \partial_{x_l} v(t,X_t) - K^X_s \partial_{x_l} v(s,X_s) \rrm_{L_p(\P)} \\
&\label{eqn:lemma:upper_bound_nabla}   
\le c_\eqref{eqn:lemma:upper_bound_nabla} \Big [  \| M_T \|_{L_p(\P)} \int_s^t\frac{ dr }{\sqrt{ T-r }} + \Big( \int_s^t \| D^2 v(r,X_r)\|^2_{L_p(\P)} dr \Big)^\frac{1}{2}  
       \Big ] 
\end{align}
with
$ c_\eqref{eqn:lemma:upper_bound_nabla} >0$ depending at most on $(\sigma,b,k,p,\P)$ (and therefore implicitly 
on $\alpha$).
\end{lemma}

\begin{proof}
Exploiting (\ref{eqn:pde_for_gradient}) and Propositions \ref{proposition:RH-A-BMO} and \ref{proposition:BMO-exponential} we get that
\equa
&   &
        \noo K^X_t \partial_{x_l} v(t,X_t) - K^X_s\partial_{x_l} v(s,X_s) \rrm_{L_p(\P)} \\
&\le&   \noo \int_s^t K^X_r (\nabla \partial_{x_l} v)(r,X_r) \sigma(r,X_r) dB_r \rrm_{L_p(\P)} \\
&   & + \bigg \|  \int_s^t K^X_r
        \bigg [  \frac{1}{2} \left | \langle \partial_{x_l}A(r,X_r),D^2(r,X_r)\rangle \right |
       + \left | \langle \partial_{x_l}b(r,X_r),\nabla v(r,X_r)^\top \rangle \right |   \\
&   &  \hspace*{15em}  + | (\partial_{x_l} k)(r,X_r) v(r,X_r)| \bigg ]  
           dr \bigg \|_{L_p(\P)} \\
&\le& b_p(\P) \noo \kla \int_s^t |K^X_r(\nabla \partial_{x_l} v)(r,X_r) \sigma(r,X_r)|^2 dr \mer^\frac{1}{2} 
      \rrm_{L_p(\P)} \\
&   & + \frac{1}{2} \| \partial_{x_l} A \|_\infty 
        \noo \int_s^t |K^X_r D^2 v(r,X_r) | dr \rrm_{L_p(\P)} \\
&   &  + \| \partial_{x_l} b \|_\infty 
        \noo \int_s^t |K^X_r \nabla v(r,X_r)| dr \rrm_{L_p(\P)} \\
&   & +  \| \partial_{x_l} k\|_\infty 
        \noo \int_s^t |K^X_r v(r,X_r)| dr \rrm_{L_p(\P)}.
\tion 
Inequality (\ref{eq:lemma1:bis}) directly yields 
\[    \sup_{r\in[0,T]} \noo K_r^X  v(r,X_r) \rrm_{L_p(\P)}
    = \sup_{r\in[0,T]} \noo \E^{\cF_r}_\Q M_T \rrm_{L_p(\P)}
   \le c_\eqref{eqn:lemma:1} \|M_T\|_{L_p(\P)}. \]
Moreover, by Lemma \ref{lemma:H1},
\[ \| \nabla v(r,X_r)\|_{L_p(\P)} 
   \leq c_{\eqref{eqn:lemma:H1}} (T-r)^{-\frac{1}{2}} \big( 2 +T^{3/2}\big)\|M_T\|_{L_p(\P)}.\]
Inserting these estimates in the above upper bound for 
\[ \noo K^X_t \partial_{x_l} v(t,X_t) - K^X_s\partial_{x_l} v(s,X_s) \rrm_{L_p(\P)}\]
gives the announced result.
\end{proof}

\begin{lemma}[{\cite[Proposition A.4]{geis:toiv:12}}]
\label{lemma:general_interpol}
Let $0 < \theta < 1$, $2 \leq q \leq \infty$ and
$d^k:[0,T)\to [0,\infty)$, $k=0,1,2$, be  measurable functions.
Assume that there are $A\ge 0$ and $D\ge 1$ such that
\equa
      \frac{1}{D} (T-t)^{\frac{k}{2}} d^k (t) 
&\le& d^0(t)
 \le D \kla \int_t^T [d^1(s)]^2 ds \mer^\frac{1}{2}, \\
      d^1(t) 
&\le& A + D \kla \int_0^t [d^2(u)]^2 du \mer^\frac{1}{2}
\tion
for $k=1,2$ and $t\in [0,T)$. Then there is a constant
$c_{(\ref{eqn:general_interpol})} >0$, depending at most on $(D,\theta,q,T)$,
such that, for $k,l\in \{ 0,1,2\}$,
\begin{equation}\label{eqn:general_interpol} 
        A + \Phi_q \left ( 
         (T-t)^{\frac{k-\theta}{2}} d^k(t) \right )
\sim_{c_{(\ref{eqn:general_interpol})}}
       A + \Phi_q \left ( (T-t)^{\frac{l-\theta}{2}} d^l(t) 
          \right ).
\end{equation}
\end{lemma}

\subsection{Proof of Theorem \ref{thm:equivalence}}
We let
\equa
d^0(t) &:= & \sqrt{T-t} + \| M_T -\ept M_T \|_{L_p(\P)},\\
d^1(t) &:= & 1 + \noo \nabla v(t,X_t) \rrm_{L_p(\P)}, \\
d^2(t) &:= & 1 + \noo    D^2 v(t,X_t) \rrm_{L_p(\P)}.
\tion
From Lemma \ref{lemma:H1} we get that
\equa
      d^1(t) 
& = & 1 + \|\nabla v(t,X_t)\|_{L_p(\P)} \\
&\le& 1+ c_{\eqref{eqn:lemma:H1}}(T-t)^{-\frac{1}{2}} \| M_T -\ept M_T \|_{L_p(\P)} 
       +c_{\eqref{eqn:lemma:H1}}(T-t) \| M_T\|_{L_p(\P)} \\
&\le&  (T-t)^{-\frac{1}{2}} [ 1 + c_{\eqref{eqn:lemma:H1}} 
                                + c_{\eqref{eqn:lemma:H1}}T\| M_T \|_{L_p(\P)}] \\
&   &  \left [ \sqrt{T-t} + \| M_T -\ept M_T \|_{L_p(\P)} \right ] \\
& = &  (T-t)^{-\frac{1}{2}} [ 1 + c_{\eqref{eqn:lemma:H1}} + c_{\eqref{eqn:lemma:H1}}T\| M_T \|_{L_p(\P)}]
       d^0(t).
\tion
From Lemma \ref{lemma:H2_new} we get that
\equa
      d^2(t)
& = & 1 + \noo D^2 v(t,X_t) \rrm_{L_p(\P)} \\
&\le& 1 +  c_{\eqref{eqn:lemma:H2_new}} \bigg [ 
      \frac{\| g(X_T) - \E_\P^{\cF_t}g(X_T)\|_{L_p(\P)}}{T-t} +  \sqrt{T-t} \| M^*\|_{L_p(\P)} \bigg ].
\tion
Using Remark \ref{remark:thm_equivalence}(5) we have that
\begin{multline*}
    \| g(X_T) - \ept g(X_T) \|_{L_p(\P)} \\
\le 2 e^{\|k\|_\infty T} \left [ \|k\|_\infty (T-t) \|M_T\|_{L_p(\P)}
    + \| M_T - \ept M_T \|_{L_p(\P)} \right ].
\end{multline*}    
Together with the previous estimate we obtain a constant $c>0$ depending at most
on $(c_{\eqref{eqn:lemma:H2_new}},k,T,\|M^*\|_{L_p(\P)})$ such that
\[ d^2(t) \le c(T-t)^{-1} d^0(t).\]
From Lemma \ref{lemma:upper_bound_diffenerence_conditional_expectations} we get that
\equa
      d^0(t)
& = & \sqrt{T-t} + \| M_T -\ept M_T \|_{L_p(\P)} \\
&\le& \sqrt{T-t} +  c_{\eqref{eqn:lemma:upper_bound_diffenerence_conditional_expectations}}
      \Big( \int_t^T \|\nabla v(s,X_s)\|_{L_p(\P)}^2ds  \Big)^\frac{1}{2} \\
&\le& [1+c_{\eqref{eqn:lemma:upper_bound_diffenerence_conditional_expectations}}]
      \Big( \int_t^T \left [ 1 + \|\nabla v(s,X_s)\|_{L_p(\P)}\right ]^2 ds  \Big)^\frac{1}{2} \\
& = & [1+c_{\eqref{eqn:lemma:upper_bound_diffenerence_conditional_expectations}}]
      \Big( \int_t^T [ d^1(s)  ]^2 ds  \Big)^\frac{1}{2}.
\tion
Finally, from Lemma \ref{lemma:upper_bound_nabla} for $s=0$ we deduce that
\equa
      d^1(t)
& = & 1 +  \|\nabla v(t,X_t)\|_{L_p(\P)} \\
&\le& 1 +  e^{\|k\|_\infty T} \|K_t^X \nabla v(t,X_t)\|_{L_p(\P)} \\
&\le& 1 +  e^{\|k\|_\infty T} \|K^X_0\nabla v(0,X_0)\|_{L_p(\P)} \\
&   &     +  e^{\|k\|_\infty T}c_{\eqref{eqn:lemma:upper_bound_nabla}} \sqrt{d} 
          \Big [  \| M_T \|_{L_p(\P)} 2\sqrt{T} + \Big( \int_0^t \|D^2 
                  v(r,X_r)\|^2_{L_p(\P)} dr \Big)^\frac{1}{2}  \Big ] \\
&\le& d_1 + d_2 \Big( \int_0^t \|D^2  v(r,X_r)\|^2_{L_p(\P)} dr \Big)^\frac{1}{2}  \\
&\le& d_1 + d_2 \Big( \int_0^t [d^2(r)]^2 dr \Big)^\frac{1}{2}           
\tion
with
\equa
d_1 &:=& 1 +  e^{\|k\|_\infty T} 
      \left [    \|K^X_0\nabla v(0,X_0)\|_{L_p(\P)} 
              +  2 c_{\eqref{eqn:lemma:upper_bound_nabla}} \sqrt{dT} 
           \| M_T \|_{L_p(\P)} \right ], \\
d_2 &:=&  e^{\|k\|_\infty T}c_{\eqref{eqn:lemma:upper_bound_nabla}} \sqrt{d}. 
\tion       

Lemma \ref{lemma:general_interpol} combined with Remark \ref{remark:thm_equivalence}(5) yields the statement. 
\qed


\end{document}